\documentclass[11pt]{amsart}

\usepackage{enumerate,url,upref,amsrefs,amssymb}

\newtheorem{theorem}{Theorem}[section]
\newtheorem{lemma}[theorem]{Lemma}
\newtheorem{proposition}[theorem]{Proposition}

\theoremstyle{definition}
\newtheorem{definition}[theorem]{Definition}
\newtheorem{example}[theorem]{Example}

\newtheorem{question}[theorem]{Question}

\theoremstyle{remark}

\numberwithin{equation}{section}

\DeclareMathOperator{\diam}{diam}
\DeclareMathOperator{\dist}{dist}
\DeclareMathOperator{\id}{id}

\newcommand{\abs}[1]{\lvert #1 \rvert}

\newcommand{\inn}[1]{\langle #1 \rangle}
\newcommand{\R}{\mathbb{R}}
\newcommand{\C}{\mathbb{C}}
\newcommand{\be}[1]{\begin{equation}\label{#1}}
\newcommand{\ee}{\end{equation}}

\title{Bi-Lipschitz embedding of projective metrics}

\author{Leonid V. Kovalev}
\address{215 Carnegie, Mathematics Department, Syracuse University, Syracuse, NY 13244}
\email{lvkovale@syr.edu}
\thanks{Supported by the NSF grant DMS-0968756.}


\begin{document}

\begin{abstract}
We give a sufficient condition for a projective metric on a subset of a Euclidean space to admit a  bi-Lipschitz embedding into Euclidean space of the same dimension. 
\end{abstract}

\maketitle

\section{Introduction}\label{intro}

A metric $d$ on a convex domain $\Omega\subseteq \mathbb R^n$ is called \emph{projective} (sometimes \emph{Desarguesian}) provided that  
the equality $d(x,z)+d(z,y)=d(x,y)$ holds if and only if $z$ is a convex combination of $x$ and $y$.  
Equivalently, a metric is projective if line segments are unique geodesics.  Two well-known classes of such metrics are strictly convex normed spaces and Hilbert geometries on convex sets~\cite{MR0054980}. A different, integral-geometric construction of projective metrics was introduced by Busemann \cite{MR0143155}.  Let $\mathcal{H}$ be the set of all hyperplanes, i.e., $(n-1)$-dimensional affine subspaces, in $\R^n$. For a set $E\subseteq \R^n$ denote by $\pi E\subseteq \mathcal{H}$   the set of all hyperplanes that intersect $E$. Throughout the paper  $n\ge 2$.

\begin{definition}\label{Busmet} Let  $\Omega\subseteq \mathbb R^n$, $n\ge 2$, be a convex domain. Suppose $\nu$ is a positive Borel measure on $\mathcal{H}$ such that 
\begin{itemize}
\item $\nu(\pi E)=0$ when $E$ is a one-point subset of $\Omega$; 
\item $\nu(\pi E)>0$ when $E$ is a   line segment in $\Omega$; 
\item $\nu(\pi E)<\infty$ for every compact set $E$ contained in $\Omega$. 
\end{itemize} 
Then 
\be{metric1}
d_\nu(x,y) = \nu(\pi [x,y])
\ee
is a \emph{Busemann-type projective metric} on $\Omega$. 
\end{definition}

 The fact that $d_\nu$ is a projective metric is immediate from the definition. In the converse direction,
Pogorelov~\cite{MR550440}, Ambartzumian~\cite{MR0372939} and Alexander \cite{MR490820} 
showed that every projective metric on $\mathbb R^2$ arises from Busemann's construction. See~\cites{MR2027175,MR0430935,Papa} for historical overview and other results towards Hilbert's 4th problem, which asks for a characterization of projective metrics.  

The fact that the Euclidean metric on $\R^n$ can be constructed as in Definition~\ref{Busmet} is a consequence of the classical Crofton formula (e.g., ~\cite{MR2455326}). We write $d_e$ for the Euclidean metric.   The main result of this 
paper is the following sufficient condition for $(\Omega,d_\nu)$ to admit a bi-Lipschitz embedding into $\R^n$.

\begin{theorem}\label{qsthm}  In the notation of Definition~\ref{Busmet}, suppose that  the identity map $\id \colon (\Omega, d_\nu) \to (\Omega,d_e)$ is 
locally $\eta$-quasisymmetric. Then $(\Omega, d_\nu)$ is bi-Lipschitz equivalent to $(\Omega',d_e)$ for some domain $\Omega'\subseteq \R^n$. Furthermore, if $\Omega=\R^n$, then $\Omega'=\R^n$. 
\end{theorem}

The assumption of Theorem~\ref{qsthm} is that there exists a modulus of quasisymmetry $\eta$ (see Definition~\ref{qsdef}) such that every point of $(\Omega,d_\nu)$ has a neighborhood in which the identity map is $\eta$-quasisymmetric. 
This is a weaker assumption than $\id$ being quasisymmetric in $\Omega$. Section~\ref{transverse} presents a more precise version of Theorem~\ref{qsthm}, namely Theorem~\ref{precise}. 

Theorem~\ref{qsthm} highlights the difference between the Busemann construction (placing a weight on 
the space of hyperplanes) and the conformal deformation (placing a weight on the Euclidean space itself). For the latter, the analog of Theorem~\ref{qsthm} fails, as was demonstrated by Semmes~\cite{MR1402671} in dimensions $n\ge 3$ and by Laakso~\cite{MR1924353} in dimension $n=2$.
In particular, Laakso constructed a nonsmooth conformal deformation of $\mathbb R^2$ such that the  resulting space is not bi-Lipschitz embeddable into any uniformly convex Banach space, despite the 
identity map to $(\mathbb R^2, d_e)$ being quasisymmetric.   

After the definitions and preliminary results are collected in Section~\ref{prel}, the proof of Theorem~\ref{qsthm} is given in \S\ref{transverse}. 
It employs a  construction of  quasiconformal maps that simultaneously extends two previously known approaches~\cites{MR0086869,MR2328811}; this connection is discussed in \S\ref{examples}. 
The concluding Section~\ref{remarks} presents some open problems.

\section{Preliminaries}\label{prel} 

Let $\Omega$ and $\nu$ be as in Definition~\ref{Busmet}.  
Fix a point $o\in\Omega$. 
For a hyperplane $H\in\mathcal{H}$  that does not pass through $o$, let 
$n(H)$ be the unit normal vector to $H$ that points out of the halfspace containing $o$. 
Define 
\be{map1}
f_\nu(x) = \int_{\pi[o,x]} n(H)\,d\nu(H),\quad x\in \Omega.
\ee
The choice of basepoint $o$ is immaterial: it contributes only an additive constant to $f_\nu$ (see the proof of Lemma~\ref{basicdist} below). Note that $f_\nu(o)=0$.

Given a nonzero vector $v$ and a hyperplane $H \in\mathcal{H}$, let $\alpha(v,H)\in [0,\pi/2]$ be the smaller angle between $H$ and the line determined by $v$. E.g., $\alpha(v,H)=\pi/2$ when $v$ is orthogonal to $H$. This notation will be used often in the sequel. 

\begin{lemma}\label{basicdist} For all $x,y\in \Omega$ 
\be{Lip1}
\abs{f_\nu(x)-f_\nu(y)}\le \nu(\pi[x,y])
\ee
and 
\be{lemqc3}
\inn{f_\nu(x)-f_\nu(y),x-y} =\abs{x-y}  \int_{\pi[x,y] } \sin \alpha(x-y, H) \,d\nu(H)
\ee
\end{lemma} 

\begin{proof} For $x,y\in \Omega$ we have 
\be{lemqc1}
f_\nu(x)-f_\nu(y) = \int_{\pi[0,x]\setminus \pi[0,y]} n(H)\,d\nu(H)
-\int_{\pi[0,y]\setminus \pi[0,x]} n(H)\,d\nu(H).
\ee
Since the symmetric difference of  $\pi[0,x]$ and $\pi[0,y]$ agrees with $\pi[x,y]$ up to a $\nu$-null set, ~\eqref{Lip1} follows. 

When $H\in \pi[0,x]\setminus \pi[0,y]$, the inner product $\inn{n(H),x-y}$ is positive. 
When $H\in \pi[0,y]\setminus \pi[0,x]$, this inner product is negative. Therefore, taking the inner product of both sides in \eqref{lemqc1} with $x-y$ yields  
\[
\begin{split}
\inn{f_\nu(x)-f_\nu(y),x-y} &= \int_{\pi[x,y] } \abs{\inn{n(H),x-y}} \,d\nu(H) 
\\ & =\abs{x-y}  \int_{\pi[x,y] } \sin \alpha(x-y, H) \,d\nu(H)
\end{split}
\]
proving~\eqref{lemqc3}. In particular, 
\be{lemqc3a}
\abs{f_\nu(x)-f_\nu(y)} \ge \int_{\pi[x,y] } \sin \alpha(x-y, H) \,d\nu(H). \qedhere
\ee
\end{proof} 

According to Lemma~\ref{basicdist},  $f_\nu$ is an injective $1$-Lipschitz map from $(\Omega,d_\nu)$ to $\R^n$. In general, it is not bi-Lipschitz. However, it satisfies a weaker noncollapsing property.

\begin{lemma}\label{cube} There is a constant $c=c(n)>0$ such that for every cube $Q\subset \Omega$ 
\be{cube0}
\diam f_\nu(Q) \ge c \, \diam_{\nu} Q 
\ee
where $\diam_\nu$ is the diameter with respect to the metric $d_\nu$. 
\end{lemma}

 \begin{proof} Let $H$ be a hyperplane that intersects $Q$ and does not meet any of its vertices. 
Write $a$ for the edgelength of $Q$. Since $Q$ contains a ball of diameter $a$, the 
 projection of $Q$ onto the line $H^\perp$ has diameter at least $a$. This diameter is realized by projections
of two vertices that are separated by $H$; call them $x$ and $y$. Since $\abs{x-y}\le a\sqrt{n}$, it follows that 
\be{angle1}
\alpha(x-y,H) \ge \sin^{-1}(1/\sqrt{n}).
\ee 
For every pair $x,y$ of distinct vertices of $Q$, let $S_{xy}$ be the set of hyperplanes $H$ that separate
$x$ from $y$ and satisfy~\eqref{angle1}.  By the above, the union of $S_{x,y}$ over all such pairs $\{x,y\}$ is  $\pi Q$. Counting the number of pairs of vertices, we conclude that there exists a pair $\{x,y\}$ such that 
$\nu(S_{x,y})\ge 4^{-n} \nu(\pi Q)$. 
For such a pair, ~\eqref{lemqc3a} yields
\be{angle2}
\abs{f_\nu(x)-f_\nu(y)} \ge 4^{-n} n^{-1/2} \nu(\pi Q) \ge  4^{-n} n^{-1/2} \diam_\nu Q. \qedhere
\ee
\end{proof} 

\begin{definition}\label{qsdef} Let $\eta\colon [0,\infty)\to [0,\infty)$ be a homeomorphism, called a \emph{modulus of quasisymmetry} below. 
A topological embedding $f\colon X\to \R^n$ of a metric space $X$ into $\mathbb R^n$ is called $\eta$-\emph{quasisymmetric} if for every  triple of distinct points $a,b,x\in X$
\begin{equation}\label{defqs}
\abs{f(x)-f(a)}\le \eta(t)\abs{f(x)-f(b)} \qquad \text{where } \quad t =\frac{d_X(x,a)}{d_X(x,b)}.
\end{equation}
\end{definition}
When there is no need to emphasize the  modulus of quasisymmetry  $\eta$, we simply say that $f$ is quasisymmetric. 
Bi-Lipschitz maps are quasisymmetric but not conversely. The foundational facts about  quasisymmetric maps in  metric spaces are presented in~\cite{MR1800917}. 

In what follows we use standard notation $B(a,r)=\{x\in\R^n\colon \abs{x-a}<r\}$ and $\overline{B}(a,r)=\overline{B(a,r)}$. Unspecified multiplicative constants $C$ and $c$ are always positive, and may 
differ from one line to another.

\section{Uniform transversality}\label{transverse}

Informally, a measure $\nu$ on $\mathcal{H}$ is uniformly transverse if it not tightly concentrated on hyperplanes that are nearly parallel to some line. The precise statement follows.

\begin{definition}\label{utdef} Let $\nu$ and $\Omega$ be as in Definition~\ref{Busmet}. We say that $\nu$ is \emph{uniformly transverse} on $\Omega$ if there exists $\kappa>0$ such that 
\be{mt0}
 \int_{\pi[x,y]} \sin \alpha(x-y,H) \,d\nu(H) \ge \kappa\, \nu(\pi[x,y]) 
\ee
for all $x,y\in\Omega$.
\end{definition}

Some remarks are in order. When a line segment $[x,y]\subset   \Omega$ is divided into subsegments, both sides of ~\eqref{mt0} are additive with respect to such partition. Thus, it suffices to verify~\eqref{mt0} for sufficiently 
short segments. Also, $\eqref{mt0}$ is equivalent to the existence of $\tau>0$ such that  
\be{mt00}
\nu(\{ H\in \pi[x,y] \colon   \alpha(x-y,H) \ge \tau \}) \ge \tau \, \nu(\pi[x,y]). 
\ee
Indeed,~\eqref{mt00} obviously implies~\eqref{mt0} with $\kappa= \tau\sin \tau$. Conversely, if~\eqref{mt0} holds, then
 letting $\tau=\kappa/2$  we find that 
\[
\begin{split}
\kappa\,\nu(\pi[x,y]) &\le  \int_{\pi[x,y]} \sin \alpha(x-y,H) \,d\nu(H)  \\ 
&\le  \tau  \,\nu(\pi[x,y]) + \nu(\{ H\in \pi[x,y] \colon  \sin \alpha(x-y,H) \ge \tau \})                                 
\end{split}
\]
hence~\eqref{mt00} holds. 

When $\Omega=\R^n$ in Definition~\ref{utdef} we simply say that $\nu$ is uniformly transverse.  The following result relates uniform transversality to the quasisymmetry of the identity map. 

\begin{proposition}\label{idprop} If  the map $\id\colon (\Omega,d_\nu) \to (\Omega, d_e)$  is  locally $\eta$-quasisymmet\-ric, then $\nu$ is uniformly transverse.  
\end{proposition} 

\begin{proof} As observed above, it suffices to consider a short segment $[x,y]$. Let $r=\abs{x-y}$. The assumption of quasisymmetry implies that by taking sufficiently small $c=c(\eta,n)>0$, 
we can ensure that any cube $Q$ with center $x$ and edgelength $cr$ satisfies 
\be{idprop1}
\diam_\nu Q \le \frac1{2\cdot 4^{n}} \,d_\nu(x,y). 
\ee
Since $\nu$-almost every hyperplane crossing $Q$ separates a pair of its vertices (and there are $2^n$ vertices), it follows that there is a pair of vertices $u,v$ such that $d_\nu(u,v)\ge 4^{-n} \nu(\pi Q)$. Thus,~\eqref{idprop1} implies  
\be{idprop2}
\nu(\pi Q) \le \frac12  \nu(\pi [x,y]).  
\ee
For any hyperplane $H\in \pi[x,y]\setminus \pi Q$ the angle $\alpha(x-y,H)$ is bounded from below by a constant that depends only on $c$. Since the set of such hyperplanes has $\nu$-measure 
at least $\frac12  \nu(\pi [x,y])$, the claim follows. 
\end{proof}

For any $\nu$ as in Definition~\ref{Busmet}, the map $f_\nu$ is   monotone in the sense that $\inn{f_\nu(x)-f_\nu(y),x-y}\ge 0$ whenever $x\ne y$; this is a consequence of~\eqref{lemqc3}.  
In fact, it satisfies a stronger property defined below. 

\begin{definition}\label{cymon} Let $\Omega$ be a convex domain in $\R^n$. A  map $f\colon \Omega\to\R^n$ is called \emph{cyclically monotone} if 
\be{cymon1}
\sum_{k=1}^m \inn{f(x_k),x_{k+1}-x_k} \le 0 
\ee
holds for all $m\ge 2$ and all $x_1,\dots,x_m\in \Omega$. Here $x_{m+1}=x_1$. 
\end{definition}

Observe that for $m=2$ the inequality~\eqref{cymon1} amounts to monotonicity. The concept of cyclic monotonicity is motivated by the fact that cyclically monotone maps are precisely 
subsets of subgradients of convex functions~\cite{MR1451876}*{Theorem 24.8}. In particular, every continuous cyclically monotone map is the gradient of a $C^1$ convex function. 

\begin{proposition}\label{cylemma} The map $f_\nu$ in~\eqref{map1} is cyclically monotone.  
\end{proposition} 

\begin{proof} Since the inequality~\eqref{cymon1} is additive with respect to $f$, it suffices to verify  it for the integrand in~\eqref{map1}. Fix a hyperplane $H$ not passing through $0$.
It can be described by the equation $H=\{x\colon \inn{x,n(H)}=c\} $ for some $c>0$.  Let $g(x)=n(H)$ if $H$ separates $x$ from $0$, and $g(x)=0$ otherwise. 
The function $U(x) = \max(c, \inn{x,n(H)} )$ is convex and its subgradient $\partial U$ satisfies $g(x)\in \partial U(x)$ for every $x\in\mathbb R^n$. Therefore, $g$ is cyclically monotone, and so is $f_\nu$. 
\end{proof} 

Yet another concept of monotonicity comes into play when $\nu$ is uniformly transverse. 

\begin{definition}\label{dmon} Let $\Omega$ be a convex domain in $\R^n$. For a fixed $\delta>0$, a  map $f\colon \Omega\to\R^n$ is called $\delta$-monotone if 
\be{dmon1}
\inn{f(x)-f(y),x-y} \ge \delta \abs{f(x)-f(y)}\,\abs{x-y} 
\ee
holds for all $x,y\in \Omega$.
\end{definition}

Neither cyclic monotonicity nor $\delta$-monotonicity imply each other. 

\begin{proposition}\label{mtlem} The following are equivalent: 
\begin{enumerate}[(i)]
\item $\nu$ is uniformly transverse; 
\item $f_\nu$ is $\delta$-monotone;
\item $f_\nu$ is a locally $\eta$-quasisymmetric embedding of $(\Omega,d_e)$ into $\R^n$.
\end{enumerate} 
The equivalence is quantitative in the sense that the constants involved in each statement  depend only on one another and on the dimension $n$. 
\end{proposition} 

\begin{proof} 
The equivalence of (i) and (ii)  is established by the identify~\eqref{lemqc3}.  By~\cite{MR2340234}*{Theorem 6}, every $\delta$-monotone map is locally quasisymmetric; 
more precisely, there exists a modulus of quasisymmetry $\eta$ that depends only on $\delta$, such that $f$ is $\eta$-quasisymmetric in every ball $B(x,r)$ such that $B(x,2r)\subset \Omega$.  
This shows (ii)$\implies$(iii).  For the converse, observe that $f_\nu$ is continuous and cyclically monotone; therefore it can be written as the gradient of a differentiable convex function $u\colon\Omega\to\mathbb R$. 
By~\cite{MR2340234}*{Lemma 18}, if  the gradient of a convex function is locally $\eta$-quasisymmetric, it is $\delta$-monotone where $\delta$ depends only on $\eta$. This completes the proof. 
\end{proof}

By virtue of Proposition~\ref{idprop}, Theorem~\ref{qsthm} is a consequence of the following more precise statement. 

\begin{theorem}\label{precise} If $\nu$ is uniformly transverse on $\Omega$, then the map $f_\nu$ defined by  ~\eqref{map1} is a bi-Lipschitz embedding of $(\Omega,d_\nu)$ into $\R^n$. 
Furthermore, if $\Omega=\R^n$ then $f_\nu(\Omega)=\R^n$. 
\end{theorem} 

\begin{proof}[Proof of Theorem~\ref{precise}] 
From~\eqref{Lip1} we see that $f_\nu$ is Lipschitz. The reverse inequality $\abs{f_\nu(x)-f_\nu(y)}\ge c\,d_\nu(x,y)$  follows by combining ~\eqref{lemqc3a} and ~\eqref{mt0}. 

By Proposition~\ref{mtlem} $f_\nu$ is a locally $\eta$-quasisymmetric embedding of $(\Omega,d_e)$ into $\R^n$, therefore it is a \emph{quasiconformal map} (e.g., ~\cite{MR1800917}). 
 It is well-known that $\R^n$ cannot be quasiconformally mapped to its proper subdomain~\cite{MR0454009}*{Theorem 17.4}. Thus, in the case  $\Omega=\R^n$ we have $f_\nu(\Omega)=\R^n$.  
\end{proof} 

\section{Examples}\label{examples}
 
A convenient way to introduce measures on  the space of hyperplanes $\mathcal{H}$ is to push them forward from a space where it is easier to construct measures. 
For example, there is a natural surjection $\R^n\times S^{n-1}\to \mathcal{H}$ given by $\Phi(a,v)=\{x\colon \inn{x,v}=\inn{a,v}\}$. Let $\omega$ be the normalized volume measure on $S^{n-1}$. 
For a Radon measure $\mu$ on $\R^n$ the pushforward $\Phi_*(\mu\times \omega)$  is a measure on $\mathcal{H}$. 

\begin{example}\label{KMW}  Let $\mu$ be a non-atomic measure (meaning $\mu(\{x\})=0$ for every $x$) such that the support of $\mu$ is not contained in any line. If
\be{tail1}
0<\int_{\R^n} \abs{x}^{-1} \,d\mu(x) <\infty,
\ee
then $\Phi_*(\mu\times \omega)$ satisfies the assumptions of Definition~\ref{Busmet}. If, in addition, $\mu$ is a doubling measure, then $\Phi_*(\mu\times \omega)$ is uniformly transverse. 
\end{example} 

Recall that a measure $\mu$ is \emph{doubling} if there exists a constant $C$ such that $\mu(B(x,r))\le C\mu(B(x,r))$ for all $x\in \R^n$ and all $r>0$. 

\begin{proof} For every $x\in\R^n$ and $r>0$ we have 
\[
\nu(\pi B(x,r)) \le \mu(B(x,r)) + C\int_{\R^n} \frac{r}{\abs{x-y}} \,d\mu(y)
\]
which implies that the first and third conditions in Definition~\ref{Busmet} hold. The second condition, $\nu(\pi[x,y])>0$, follows from the support of $\mu$ not being contained in the line through $x$ and $y$.   

Suppose $\mu$ is doubling. Fix distinct points $x$ and $y$ and let $r=\abs{x-y}$. Also fix a unit vector $w$ that is orthogonal to $x-y$. 
For $k=1,2,\dots$ let $\mu_k$ be the restriction of $\mu$ to the spherical shell $A_k=B(x,2^{k+1}r)\setminus B(x, 2^k r)$.
This shell contains the open ball $B_k=B(x+3\cdot 2^{k-1}w , 2^{k-1})$. The doubling condition implies that 
$\mu (B_k) \ge c  \mu(A_k)$ with $c$ independent of $k$. It is geometrically evident that  every hyperplane $H$ that meets both $[x,y]$ and  $B_k$ satisfies $\alpha(x-y,H) \ge \pi/4$.  
Thus, the measure $\nu_k = \Phi_*(\mu_k\times \omega)$ satisfies 
\be{KMW5}
 \int_{\pi[x,y]} \sin \alpha(x-y,H) \,d\nu_k(H) \ge c\, \nu_k(\pi[x,y]) 
\ee
  with $c$ independent of $k$.  Observe also that the restriction of $\mu$ to $B(x,2r)$, which is not included in any $\mu_k$, is comparable in mass to $\mu_1$; thus its contribution to $\nu(\pi[x,y])$ is controlled by 
~\eqref{KMW5} with $k=1$. Summing  over $k$, we conclude that $\nu$ is uniformly transverse. 
\end{proof} 

When $\nu=\Phi_*(\mu\times \omega)$, the formula~\eqref{map1} yields
\be{KMW3}
f_\nu(x) =c  \int_{\R^n} \left(\frac{x-y}{\abs{x-y}} + \frac{y-o}{\abs{y-o}}\right) \,d\mu(y)
\ee
with some constant factor $c>0$. Indeed, it suffices to verify~\eqref{KMW3} for a unit point mass $\delta_a$ because general measures can be approximated by  linear 
combinations of point masses. In turn, $\delta_a$ is the limit of normalized restrictions of the Lebesgue measure to  $B(a,r)$ as $r\to 0$. 
If the basepoint $o$ in the definition of $f_\nu$ coincides with $a$, a symmetry consideration yields
\[
\int_{\pi[0,x]} n(H)\,d\nu(H) = c\, \frac{x}{\|x\|},\quad \abs{x}>r 
\]
with $c$ independent of $x$ or $r$. Changing the basepoint    $o$ contributes additive constants to $f_\nu$ and to the right side of~\eqref{KMW3}. 
Since both sides of~\eqref{KMW3} turn to $0$ when $x=o$,  the additive constants agree. This proves ~\eqref{KMW3}.

The integral~\eqref{KMW3}  was used in ~\cite{MR2328811} to construct quasiconformal maps from doubling measures.
Thus, Example~\ref{KMW} shows that the results of \S\ref{transverse} recover some of the main results of~\cite{MR2328811}. 

Beurling and Ahlfors~\cite{MR0086869} proved that every quasisymmetric self-map of $\R$ extends to a quasisymmetric self-map of $\R^2$. Up to orientation, quasisymmetric maps on a line are precisely indefinite integrals of 
doubling measures of $\R$. The following proposition shows that the Beurling-Ahlfors extension can be obtained from Theorem~\ref{precise}.

\begin{example}\label{BA}  Let $\mu$ be a doubling measure on the real axis $\R$ of the complex plane $\C\approx \R^2$. Denote by $\widetilde{\omega}$   the restriction of the arclength measure on $S^{1}$
to the set of unit vectors $(v_1,v_2)$ such that $v_2 \ge \frac{\sqrt{3}}{2}$.  Let $\nu=\Phi_*(\mu\times \widetilde{\omega})$ with $\Phi$ as above. Then $\nu$ is uniformly transverse. 

Moreover, 
$f_\nu\colon  \R^2\to\R^2$ is a quasiconformal map such that $f_\nu(\R)  =  \R$ and for all $s,t\in\R$, $s<t$,  we have $ f_\nu(t)-f_\nu(s)=\mu([s,t])$. 
\end{example} 

Observe that every line $H$ in the support of $\nu$ crosses $\R$ at an angle of at least $\pi/3$. 

\begin{proof} Fix two distinct points $x$ and $y$. 
If the angle that $x-y$ forms with the real axis is less than $\pi/4$, the uniform transversality condition holds for the segment $[x,y]$ by the construction of $\nu$. 
Suppose that this angle is at least $\pi/4$. By partitioning the segment $[x,y]$, we may assume that $\dist([x,y],\R)\ge  \abs{x-y}$.
Also without loss of generality, $\dist(y,\R)>\dist(x,\R)$.

Let $I$ be the segment on $\R$ formed by the intersection points of $\R$ with the lines that meet $[x,y]$ at an angle less than $\pi/12$. 
Note that $I$ is the base of a triangle with vertex $y$ in which the angle at $y$ is $\pi/6$ and the segment $[x,y]$ bisects this angle. 
Let $p$ be the nearest endpoint of $I$  to $y$; if the endpoints are equidistant from $y$ (i.e., $[x,y]$ is vertical), pick either one. Let $I'\subset\R$ be the segment of the same length as $I$ and such that $I\cap I'=\{p\}$. 
The doubling condition implies $\mu(I')\ge c\mu(I)$. 
It follows that the restriction of $\mu$ to $I'$ is responsible for a certain fraction of $\nu(\pi [x,y])$; and since the lines that intersect both $I'$ and $[x,y]$  form the angle 
of at least $\pi/12$ with the latter, the measure $\nu$ is uniformly transverse. 

The quasiconformality of $f_\nu$ follows from Theorem~\ref{precise}. The fact that $f_\nu(\R)=\R$ is a consequence of the symmetry of $\nu$: reflection of the plane across the real axis leaves $\nu$ invariant. 
Finally, for real $s<t$ the definition of $f_\nu$ yields 
\[
f_\nu(t)-f_\nu(s) = \int_s^t \int_{-\pi/6}^{\pi/6} \cos\theta \, d\theta \, d\mu = \mu([s,t]).  \qedhere
\]
\end{proof} 

\section{Concluding remarks}\label{remarks} 

Theorem~\ref{qsthm} leads to several natural questions. The main result of Pogorelov's book~\cite{MR550440} is that sufficiently smooth projective metrics on $\R^3$ can be obtained as $d_\nu$ with $\nu$ being a \emph{signed}
measure on $\mathcal{H}$. Szab\'o~\cite{MR835025} extended this result to all dimensions.  Although the definition of our map $f_\nu$ makes sense when $\nu$ is a  signed measure,   all  results  of this paper rely on $\nu$ being positive. 

\begin{question}
Can Theorem~\ref{qsthm} be extended to signed measures $\nu$ that generate positive metrics $d_\nu$? 
\end{question} 

A well-known \emph{necessary} condition for a metric space $X$ to have a bi-Lipschitz embedding into a Euclidean space is that $X$ is doubling, but this 
 condition is not sufficient in general~\cite{MR1800917}.  A projective metric need not be doubling. For example, the Beltrami-Klein model of the hyperbolic space is a non-doubling projective metric on the unit ball
of $\R^n$, since the hyperbolic space fails the doubling condition.  More generally, Hilbert geometries on convex domains are typically Gromov hyperbolic~\cites{MR2010741,MR1923418}. 

\begin{question}
Does every doubling projective metric on a convex domain $\Omega\subseteq  \R^n$ admit a bi-Lipschitz embedding into some $\R^N$? Or even into $\R^n$? 
\end{question}

\end{document}